\documentclass{article}

\usepackage{amsmath,amsthm,amsfonts,amssymb,amscd,cite,graphicx,subcaption}
\usepackage{authblk}

\newtheorem{definition}{Definition}[section]
\newtheorem{theorem}{Theorem}[section]
\newtheorem{remark}{Remark}[section]

\begin{document}

\title{On Directed Graphs with the Same Sum over Arborescence Weights}

\author[1,2]{Sayani Ghosh}
\author[3]{Bradley S. Meyer}

\affil[1]{Schmid College of Science and Technology, Chapman University,
Orange, CA, USA} 
\affil[2]{Institute of Quantum Studies, Chapman University,
Orange, CA, USA}
\affil[3]{Department of Physics and Astronomy, Clemson University,
Clemson, SC, USA}

\maketitle








 \begin{abstract}
 \noindent
We show that certain digraphs with the same vertex set but different
arc sets have the same sum over the weights of all arborescences with a
given root vertex.  We relate our results to the Matrix-Tree Theorem and
show how they
provide a graphical approach for factoring matrix determinants.
 \\[2mm]
 \end{abstract}

\section{Introduction}

A directed graph (digraph) $\Gamma$ is a set of vertices and a set of arcs,
which are ordered pairs of vertices.
In particular, an arc $e = (i,j)$
is an arrow directed from vertex $i$ to vertex $j$, where $i$ and $j$ are
both elements of the vertex set of $\Gamma$.  The source of this arc $e$ is
$s(e) = i$ while the target is $t(e) = j$.
An arborescence on $\Gamma$ is a subgraph of $\Gamma$
that is a spanning directed tree.  The {\em root vertex}
of the arborescence is the one vertex in the subgraph
that has indegree zero.  All other vertices have indegree exactly
one.  If the digraph is weighted, an arborescence
has weight equal to the product of the weights of the arcs in the
arborescence.

We next define $\Gamma_v$ to be a digraph
such that vertex $v$ has no in arcs; that is, no arc in $\Gamma_v$ has
$v$ as its target.  Since $\Gamma_v$ has
no in arcs to $v$, and since an arborescence is a subgraph of $\Gamma_v$
that is a spanning directed tree,
any arborescence in $\Gamma_v$ must have $v$ as the root vertex.
We thus call $v$ the root vertex of the digraph $\Gamma_v$.

In this paper, we show that certain digraphs with the same vertex set but different arc sets have the same sum over the weights of all arborescences.  While the graphs we investigate are only a small subset of all graphs with the same sum over arborescence weights, they are easily identified from the graphs themselves.  By the Matrix-Tree Theorem, this enables straightforward manipulation of a weighted digraph while leaving the determinant of the corresponding matrix unchanged from that corresponding to the original graph.

\section{Main Results}

Certain digraphs with arcs having the same target and weight but different sources can have the same sum over arborescence weights.

\begin{theorem}[Moving-Arc Theorem]
Consider a directed graph $\Gamma_v$ with vertex set $V$ and
arc set ${\cal A}$ and with root vertex $v \in V$.
Consider further that $\{a, b, c\} \in V$ and
that ${\cal A} = D \cup e$, where $D$ is a set of arcs and $e$ is an arc
with weight $w(e)$ such that $s(e) = a$ and $t(e) = b$.
Suppose another digraph $\Gamma_v'$ has vertex set $V' = V$,
root vertex $v \in V'$,
arc set ${\cal A'} = D' \cup e'$ such that $D' = D$, $s(e') = c$,
$t(e') = b$, and $w(e') = w(e)$.
The sum of all arborescence weights in $\Gamma_v'$ will be the same as
the sum of all arborescence weights in $\Gamma_v$ if $a$ and $b$ are not
strongly connected in $\Gamma_v$ and $c$ and $b$ are not strongly connected
in $\Gamma_v'$.
\label{theorem:move}
\end{theorem}

\begin{proof}
Since $e \in {\cal A}$ and $e' \in {\cal A'}$, and since $t(e) = b$ and
$t(e') = b$, $b$ is not the root vertex of $\Gamma_v$ or $\Gamma_v'$;
thus,
an arborescence in $\Gamma_v$ and in $\Gamma_v'$ must include an in arc to $b$.
That in arc may or may not be $e$ in $\Gamma_v$ and $e'$ in $\Gamma_v'$.
Consider the set of arborescences in $\Gamma_v$ that do not include $e$.
All arcs in these arborescences come only from arc set $D$.
This set of arborescences is identical to the set of arborescences in
$\Gamma_v'$ that do not include $e'$ since all arcs in these arborescences
come only from arc set $D'$, which is identical to $D$.

Now consider an arborescence $S$ in $\Gamma_v$ that includes $e$.
The arcs in $S$ are $H \cup e$, where $H \subset D$.
There is a unique corresponding subgraph $S'$ in $\Gamma_v'$
with arcs $H \cup e'$.  Since the only difference between $S$ and $S'$
will be the in arc to $b$,
and since the indegree of vertex $b$ is one in both $S$ and $S'$,
$S'$ will be an arborescence if there is no cycle in $S'$.
Since $S$ is an arborescence, there is no cycle among the set of arcs $H$.
$e'$ provides a path from $c$ to $b$;
thus, if there is no path from $b$ to $c$ in $S'$, $S'$ will be an arborescence.
If there is no path from $b$ to $c$ in $\Gamma_v'$,
then there can be no path
from $b$ to $c$ in $S'$ and, thus, $S'$ is guaranteed to be an arborescence.
Thus, if $b$ and $c$ are not strongly connected in $\Gamma_v'$,
for each arborescence $S$ in $\Gamma_v$ that includes $e$,
there will be a unique corresponding arborescence $S'$ in $\Gamma_v'$
that includes $e'$.

Consider now an arborescence $S'$ in $\Gamma_v'$ that includes $e'$.
There is a subgraph $S$ in $\Gamma_v$ with $e$ replacing $e'$.
By the same argument as in the previous paragraph, since $s(e) = a$
and $t(e) = b$, $S$ is guaranteed to be an arborescence in $\Gamma_v$
if $a$ and $b$ are not strongly connected in $\Gamma_v$.
Thus, if $a$ and $b$ are not strongly connected in $\Gamma_v$,
for each arborescence $S'$ in $\Gamma_v'$ that includes $e'$,
there will be a unique arborescence $S$ in $\Gamma_v$ that includes $e$.

These arguments establish that if $a$ and $b$ are not strongly connected
in $\Gamma_v$ and $c$ and $b$ are not strongly connected in $\Gamma_v'$,
then there is a one-to-one correspondence between arborescences in $\Gamma_v$
and $\Gamma_v'$.
Arborescences in $\Gamma_v$ that do not include $e$ have an identical
corresponding arborescence in $\Gamma_v'$.  These arborescences have the
same weight.  For each arborescence in $\Gamma_v$ that includes $e$,
there is a unique corresponding arborescence in $\Gamma_v'$ that includes $e'$.
Since these arborescences only differ
in the arcs $e$ and $e'$, and since $w(e') = w(e)$, these arborescences
have the same weight.  Thus, each arborescence in $\Gamma_v$ has the
same weight as its corresponding arborescence in $\Gamma_v'$ and
the sum of the arborescence weights is thus the same in $\Gamma_v$
and $\Gamma_v'$.
\end{proof}

\begin{remark}
In Theorem \ref{theorem:move}, $\Gamma_v$ and $\Gamma_v'$ are
distinct digraphs.  It is conceptually convenient, however,
to consider $\Gamma_v'$ as the same as $\Gamma_v$ except that we have
moved the source of arc $e$ from vertex $a$ to vertex $c$ but kept
the target vertex $b$ the same.
In this way, Theorem \ref{theorem:move} shows that we leave the
sum of arborescence weights of a digraph $\Gamma_v$ with root vertex $v$
unchanged when we move the source of arc $e$ to a new source,
as long as the source vertex and target vertex of $e$
are not strongly connected in the digraph before and after the move.
This is the origin of our name for the theorem.
\label{remark:move}
\end{remark}


A directed graph may have parallel arcs.  Those arcs can be combined to
leave the sum over arborescence weights unchanged.

\begin{theorem}[Combining-Arcs Theorem]
Consider a digraph $\Gamma_v$ with vertex set $V$, arc set ${\cal A}$,
and root vertex $v \in V$.  Consider further
that there are parallel arcs $e_1 \in {\cal A}$
and $e_2 \in {\cal A}$ with weights $w(e_1)$ and $w(e_2)$
with $s(e_1) = s(e_2) = a \in V$ and $t(e_1) = t(e_2) = b \in V$.
From this, ${\cal A} = D \cup \{e_1, e_2\}$; thus, $D$ is the set of
all arcs in ${\cal A}$ other than $e_1$ and $e_2$.
Now consider a second digraph $\Gamma_v'$ with vertex set $V' = V$,
root vertex $v \in V'$, and arc set ${\cal A} = D' \cup e_c$ with $D' = D$ and
$s(e_c) = a$, $t(e_c) = b$, and $w(e_c) = w(e_1) + w(e_2)$.
The sum over all arborescences in $\Gamma_v$ is equal to the sum over all arborescences in  $\Gamma_v'$.
\label{theorem:add}
\end{theorem}
\begin{proof}
Consider an arborescence $S_1$ in $\Gamma_v$ that consists of arcs
$H \cup e_1$, where $H \subset D$.  There is exactly one other arborescence
$S_2$ in $\Gamma_v$ that consists of arcs $H \cup e_2$.  These arborescences
contribute weight $W(H) (w(e_1) + w(e_2))$ to the sum over all arborescences
in $\Gamma_v$, where $W(H)$ is the product of the weights of all arcs in $H$.
For each such pair of arborescences in $\Gamma_v$, there
is exactly one arborescence $S'$ in $\Gamma_v'$ with arcs $H \cup e_c$ and
weight $W(H) w(e_c)$.  The combined weights of $S_1$ and $S_2$ equal that
of $S'$ since $w(e_c) = w(e_1) + w(e_2)$.  For each such $H$ in $\Gamma_v$
and $\Gamma_v'$,
the combined weight of $S_1 = H \cup e_1$ and $S_2 = H \cup e_2$
equals the weight of the corresponding
$S' = H \cup e_c$ in $\Gamma_v'$; thus,
the sum over all arborescence weights with combined arcs $e_1$ and $e_2$
in $\Gamma_v$ equals the sum over all arborescence weights with arc $e_c$
in $\Gamma_v'$.  The sum over all arborescence weights in $\Gamma_v$ not
involving $e_1$ or $e_2$ is exactly the same as the sum over all arborescences
in $\Gamma_v'$ not involving $e_c$; thus, the sum over all arborescence
weights in $\Gamma_v$ is equal to the sum over all arborescence weights in
$\Gamma_v'$.
\end{proof}

\begin{remark}
    The digraphs $\Gamma_v$ and $\Gamma_v'$ in Theorem \ref{theorem:add} are distinct.  As in Remark \ref{remark:move}, it is conceptually convenient to think of $\Gamma_v'$ as the same as $\Gamma_v$ except that we have combined arcs $e_1$ and $e_2$ into a single arc with weight equal to the sum of the weights of the two original arcs.  We thus call Theorem \ref{theorem:add} the Combining-Arcs Theorem.
    \label{remark:combine}
\end{remark}

\section{Relation to the Matrix-Tree Theorem}

The matrix-tree theorem relates sums over arborescences in digraphs
to matrix determinants (e.g., \cite{chen2012applied,doi:10.1137/19M1265193,MOON1994163}).  For weighted digraphs, we state the theorem
in the following form \cite{ghosh2023digraph}.
\begin{theorem}
Let $A = [a_{ij}]$ denote the $n \times n$ matrix
such that
\begin{equation}
a_{ij} = 
\begin{cases}
-u_{ij}, & i\ne j, 1 \leq i, j \leq n\\
 \sum_{k = 1}^n u_{kj}, & i = j, 1 \leq i \leq n
\end{cases}
\label{eq:aij}
\end{equation}
Construct a digraph $\Gamma_0$ with $n + 1$ vertices labeled $0, 1, ..., n$
with $0$ the root vertex.
An arc $e = (i,j)$ in the graph has weight $w(e) = u_{ij}$ for $i \ne 0$ while
an arc $e = (0,j)$ has weight $w(e) = u_{jj}$.  Then
\begin{equation}
det(A) = \sum_{B \in S(\Gamma_0)} W(B)
\end{equation}
where $S(\Gamma_0)$ is the set of all arborescences in $\Gamma_0$ and
\begin{equation}
W(B) = \prod_{e \in B} w(e)
\label{theorem:matrix-tree}
\end{equation}
\end{theorem}

\begin{remark}
By Remark \ref{remark:move}, we leave the sum over all arborescence
weights unchanged
when we move an arc $e$ in the digraph $\Gamma_0$ if the
source and target of $e$ are not strongly connected before and after the move.
By Theorem \ref{theorem:matrix-tree}, the sum over all arborescence weights
in digraph $\Gamma_0$ is the determinant of the matrix corresponding to
$\Gamma_0$.  Thus, the determinant of the matrix corresponding to the
digraph $\Gamma_0$ is left unchanged when we move an arc $e$ in the digraph
if the source and target of $e$ are not strongly connected before and
after the move.
\label{remark:correspond}
\end{remark}

\begin{remark}
Matrix determinants are left unchanged by addition of a factor times
a row or column to another row or column
in the original matrix; thus, Remark \ref{remark:correspond} indicates that
moving an arc in $\Gamma_0$ such that the determinant of the corresponding
matrix $A$ remains unchanged corresponds to a set of row or column additions
in $A$.
\label{remark:final_move}
\end{remark}

\begin{remark}
The Laplacian matrix $L$ for a graph is defined as $L=D-A'$, where $D$ is the diagonal degree matrix and $A'$ is the adjacency matrix. In weighted multi-digraphs, each element $L_{ij}$ is
formed by summing the weights of arcs that have source $i$ and target $j$ \cite{chebotarev2006matrixforest}.  Elements $L_{ii}$ include contributions from loops or, in our version of the matrix-tree theorem, from arcs $(0, i)$.  Theorem \ref{theorem:add} implies the invariance of $det(L)$ under merging of parallel arcs in the graph described by $L$ and justifies the merging operation for matrix-tree-theorem and matrix-forest-theorem analyses \cite{chebotarev2006matrixforest, doi:10.1137/19M1265193}.

\label{remark:Laplacian}    
\end{remark}

As an example of our results, consider the upper-triangular matrix $M$ given
by
\begin{equation}
    A = \begin{pmatrix}
u_{11} & -u_{12} & -u_{13}\\
0 & u_{12} + u_{22} & -u_{23} \\
0  & 0 & u_{13} + u_{23} + u_{33}
\end{pmatrix}
\label{eq:A_ex}
\end{equation}
The determinant is the product of the diagonal elements: $u_{11}(u_{12} + u_{22})(u_{13} + u_{23} + u_{33})$.
The digraph $\Gamma_0$ corresponding to the matrix is shown in
Fig. \ref{fig:upper-triangular-digraph}.
\begin{figure}[!h]
   \centering
   \includegraphics[width=0.25\textwidth]{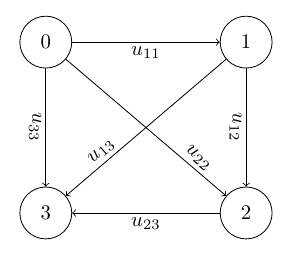}
    \caption{The matrix digraph $\Gamma_0$ corresponding to matrix $A$ in Eq. (\ref{eq:A_ex}).}
    \label{fig:upper-triangular-digraph}
\end{figure}

No two vertices in the graph in Fig. \ref{fig:upper-triangular-digraph} are
strongly connected.  It is therefore possible by Remark \ref{remark:move}
to move the source of any arc to the root vertex 0.
By Remark \ref{remark:combine} we may then combine each set of parallel
arcs into a single
arc with weight equal to the sum of the individual arc weights.
The resulting digraph
is shown in Fig. \ref{fig:upper-moved-digraph}.  The product of the
arc weights gives the expected determinant.
\begin{figure}[!h]
   \centering
   \includegraphics[width=0.4\textwidth]{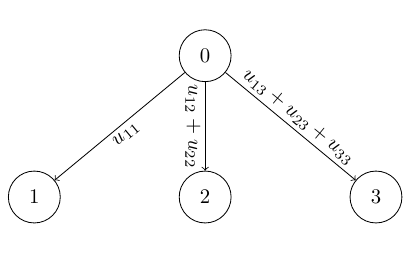}
    \caption{The matrix digraph $\Gamma_0$ after moving arcs.}
    \label{fig:upper-moved-digraph}
\end{figure}

The digraph in Fig. \ref{fig:upper-moved-digraph} corresponds to a diagonal
matrix $A'$ with elements $a_{11}' = u_{11}$, $a_{22}' = u_{12} + u_{22}$,
and $a_{33}' = u_{13} + u_{23} + u_{33}$.  Matrix $A'$ may be obtained from
matrix $A$ through three column sum operations.  The first is to add the
factor $u_{12} / u_{11}$ times the first column to the second column in $A$.
Two more such operations produce $A'$.
This illustrates Remark \ref{remark:final_move}.

\begin{remark}
    By Remark \ref{remark:final_move}, any digraph corresponding to a matrix related to the original matrix by a series of row or column operations will have the same sum over arborescence weights as the digraph corresponding to the original matrix.  In general, a single row or column operation would correspond to changing the weights of several arcs in the digraph.  In Theorems \ref{theorem:move} and \ref{theorem:add} (and Remarks \ref{remark:move} and \ref{remark:combine}), we focus on digraphs obtained from the original digraph by moving the source of a single arc or combining a single pair of parallel arcs.
    \label{remark:other_related_digraphs}
\end{remark}

\section{Graphical Calculation of a Matrix Determinant}

In this section we consider a graphical vertex-isolation approach that uses Remarks \ref{remark:move} and \ref{remark:combine} to compute a matrix determinant graphically from Theorem \ref{theorem:matrix-tree}.
We begin with a definition.

\begin{definition}
    Consider a digraph $\Gamma_0$ with root vertex 0.  A subgraph of $\Gamma_0$ is {\bf rooted} at vertex $v$ if it has arc $(0,v)$ and there are no other in arcs to $v$.
\end{definition}

Each arborescence in the sum over all arboresences of the digraph $\Gamma_0$ corresponding to a non-singular matrix must be rooted at one or more vertices.   Consider now two subgraphs of $\Gamma_0$.  The first is $\Gamma_0^{(1)}$, which is rooted at vertex 1. The second is $\Gamma_0^{({\bar 1})}$, which is not rooted at vertex 1.  To contribute arborescences, $\Gamma_0^{({\bar 1})}$ must, of course, be rooted at some other vertex or vertices than $1$.  All arborescences of $\Gamma_0$ that are rooted at vertex $1$ are arborescences of $\Gamma_0^{(1)}$ since no arborescence of $\Gamma_0^{({\bar 1})}$ can include the arc $(0, 1)$.  Similary all arborescences of $\Gamma_0$ that are not rooted at $1$ must be arborescences of $\Gamma_0^{({\bar 1})}$ since arborescences of $\Gamma_0^{(1)}$ must include the arc $(0, 1)$.  This means the sum over arborescences of $\Gamma_0$ will be equal to the sum of arborescences over $\Gamma_0^{(1)}$ and $\Gamma_0^{({\bar 1})}$.

We may proceed further converting $\Gamma_0^{({\bar 1})}$ into two new digraphs, $\Gamma_0^{(2)}$ and $\Gamma_0^{({\bar 2})}$.  $\Gamma_0^{(2)}$ is rooted at vertex 2 and not rooted at vertex 1, since it derives from $\Gamma_0^{({\bar 1})}$.  $\Gamma_0^{({\bar 2})}$ is not rooted at vertices 1 and 2.  We repeat this procedure until we have $n$ digraphs $\Gamma_0^{(1)}$, ..., $\Gamma_0^{(n)}$.  The digraph $\Gamma_0^{(j)}$ is rooted at vertex $j$ and not rooted at vertices $i$ such that $i < j$.
This means the sum over arborescence weights of $\Gamma_0$ will be equal to the sum of arborescence weights of $\Gamma_0^{(1)}$, $\Gamma_0^{(2)}$, ..., $\Gamma_0^{(n)}$.

We may further factor by an isolation procedure.  Consider the digraph $\Gamma_0^{(j)}$.  Since it is rooted at vertex $j$, vertex $j$ is not strongly connected to any other vertex in $\Gamma_0^{(j)}$.  By Remark \ref{remark:move}, the source of any arc $(j, k)$ may thus be moved from vertex $j$ to the root arc 0 since that vertex is not strongly connected to any other vertex in the graph.  The result is that vertex $j$ is now isolated (no out arcs and the only in arc is $(0, j)$).  If vertex $k > j$, it may have initially been rooted in $\Gamma_0^{(j)}$, so there are now two arcs $(0, k)$, the original one and the one that was moved.  By Remark \ref{remark:combine}, combine these into a single arc with weight equal to the sum of the weights of the two arcs.  If vertex $k < j$, it was not initially rooted in $\Gamma_0^{(j)}$, so simply move $(j, k)$ to $(0, k)$ and leave the weight the same.  Doing this for all $k \in \{1, ..., n\}, k \ne j$ leaves a modified $\Gamma_0^{(j)}$ that now has isolated vertex $j$ and rooted vertices $k \ne j$.  This digraph may be split into two digraphs: digraph $\Gamma_0^{(j),(1)}$ that is isolated at vertex $j$ and rooted at vertex 1, and $\Gamma_0^{(j),({\bar 1})}$ that is isolated at vertex $j$ and not rooted at vertex 1.  One can proceed as before and generate the $n-1$ digraphs $\Gamma_0^{(j), (k)}$ with $k \ne j$.  For $\Gamma_0^{(j), (k)}$, isolate vertex $k$ by moving and combining arcs.  Repeat this procedure until all digraphs are fully isolated (that is, only have arcs $(0, k)$ for all vertices $k \ne 0$).  There will be $n!$ such digraphs, and the sum over weights of these arborescences will be the determinant of the original matrix.

As an example of the vertex-isolation procedure, consider the following matrix $M$:
\begin{equation}
    M = \begin{pmatrix}
u_{11} + u_{21} + u_{31} & -u_{12} & -u_{13}\\
-u_{21} & u_{12} + u_{22} + u_{32} & -u_{23} \\
-u_{31}  & -u_{32} & u_{13} + u_{23} + u_{33}
\end{pmatrix}
\label{eq:M}
\end{equation}
The digraph $\Gamma_0$ corresponding to the matrix is shown in
Fig. \ref{fig:matrix-digraph}.
\begin{figure}[!h]
   \centering
   \includegraphics[width=0.35\textwidth]{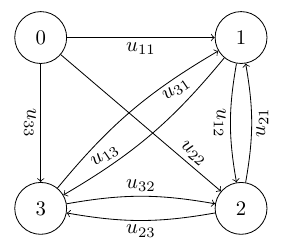}
    \caption{The matrix digraph $\Gamma_0$ corresponding to matrix $M$ in Eq. (\ref{eq:M}).}
    \label{fig:matrix-digraph}
\end{figure}

To compute the determinant of $M$, first root $\Gamma_0$
at vertex 1.  This results in digraphs $\Gamma_0^{(1)}$ and
$\Gamma_0^{(\bar{1})}$.  From $\Gamma_0^{(\bar{1})}$, derive
$\Gamma_0^{(2)}$ and $\Gamma_0^{(\bar{2})}$.  Finally from
$\Gamma_0^{(\bar{2})}$, derive $\Gamma_0^{(3)}$.  The resulting digraphs are
shown in Fig. \ref{fig:explicit}.

\begin{figure}[!h]
    \centering
    \begin{subfigure}[b]{0.25\textwidth}
        \centering
        \includegraphics[width=\textwidth]{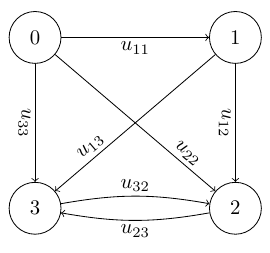}
        \caption{$\Gamma_0^{(1)}$}
        \label{fig:Gamma_01}
    \end{subfigure}
    \hfill
    \begin{subfigure}[b]{0.25\textwidth}
        \centering
        \includegraphics[width=\textwidth]{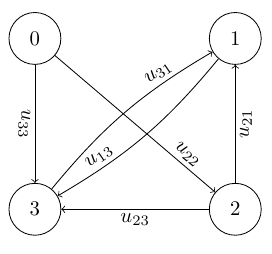}
        \caption{$\Gamma_0^{(2)}$}
        \label{fig:Gamma_02}
    \end{subfigure}
    \hfill
    \begin{subfigure}[b]{0.25\textwidth}
        \centering
        \includegraphics[width=\textwidth]{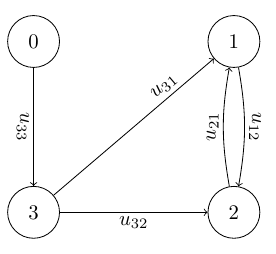}
        \caption{$\Gamma_0^{(3)}$}
        \label{fig:Gamma_03}
    \end{subfigure}
    \caption{The digraphs resulting from sequential rooting of $\Gamma_0$ in Fig. \ref{fig:matrix-digraph}.}
    \label{fig:explicit}
\end{figure}

\begin{figure}[!h]
    \centering
    \begin{subfigure}[b]{0.25\textwidth}
        \centering
        \includegraphics[width=\textwidth]{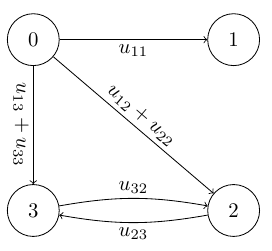}
        \caption{$\Gamma_0^{(1)}$}
        \label{fig:Gamma_01}
    \end{subfigure}
    \hfill
    \begin{subfigure}[b]{0.25\textwidth}
        \centering
        \includegraphics[width=\textwidth]{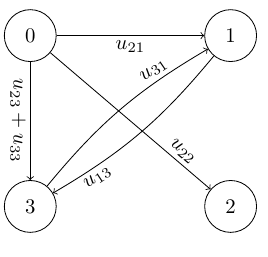}
        \caption{$\Gamma_0^{(2)}$}
        \label{fig:Gamma_02}
    \end{subfigure}
    \hfill
    \begin{subfigure}[b]{0.25\textwidth}
        \centering
        \includegraphics[width=\textwidth]{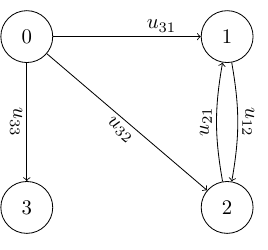}
        \caption{$\Gamma_0^{(3)}$}
        \label{fig:Gamma_03}
    \end{subfigure}
    \caption{The digraphs of Fig. \ref{fig:explicit} after the isolation procedure.}
    \label{fig:isolation}
\end{figure}

\begin{figure}[!h]
   \centering
   \includegraphics[width=\textwidth]{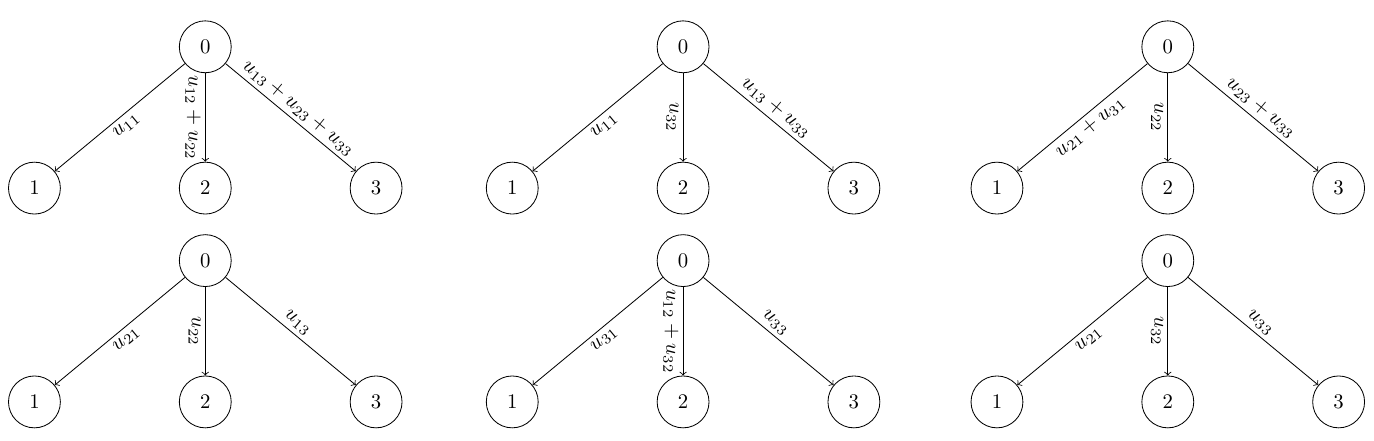}
    \caption{The fully isolated digraphs.}
    \label{fig:isolated}
\end{figure}

Once we have created the rooted subgraphs of the original matrix digraph $\Gamma_0$, we perform the isolation
procedure on each subgraph by finding the arcs whose source and target are not strongly
connected and moving the source to vertex 0 and combining parallel arcs.  The resulting digraphs are
shown in Fig. \ref{fig:isolation}.  We then carry out subsequent
rooting and isolation of vertices in each of the subgraph $\Gamma_0^{(1)}$,
$\Gamma_0^{(2)}$, and $\Gamma_0^{(3)}$.

The $3! = 6$ fully isolated digraphs resulting from the isolation procedure described above are shown in Fig. \ref{fig:isolated}.  
The resulting determinant is the sum over the weights of the six isolated graphs in Fig. \ref{fig:isolated}:
\begin{equation}
    \begin{split}
    det(M) &= u_{11}(u_{12}+u_{22})(u_{13}+u_{23}+u_{33}) + u_{11} u_{32}(u_{13}+u_{33}) \\
    &+ (u_{21}+u_{31})u_{22}(u_{33} + u_{23}) + u_{21} u_{22} u_{13}\\
    &+ u_{31}(u_{12}+u_{32})u_{33} + u_{21}u_{32}u_{33} 
\end{split}
\label{eq:isolation}
\end{equation}

In carrying out the isolation procedure, we always chose to isolate first the vertex with the smallest index.  Other fully isolated digraphs would result from a different order of vertex isolating, which would result in a sum over different terms.  Nevertheless, the sum over all isolated digraph weights, the determinant, would still be the same.

An alternative to the sequential rooting strategy above would be to create all the possible rooted subgraphs directly. Each of these subgraphs would be rooted at some set of vertices and not rooted at the remaining vertices.  Any arborescence of the original graph would derive uniquely from one of these subgraphs since it would have a unique set of rooted vertices.  To compute the number of rooted subgraphs, we partition the vertices into two non-empty subsets, one of rooted vertices and one of non-rooted vertices.  The root vertex 0 itself would belong to the latter subset since the graph has no loops.  Thus, the number of such partitions, and, hence, the number of rooted subgraphs, would be the Stirling number of the second kind $S(n+1, 2)$, where $n+1$ is the number vertices in the graph (including the root vertex 0).  This is $2^n - 1$.  For each of these subgraphs, one would move the source of arcs originating at the rooted vertices to the vertex 0 and combine them with the existing arcs.  This would isolate the rooted vertices.  One would then create all the rooted subgraphs of this subgraph, isolate each of those subgraphs, and then repeat until the resulting subgraphs are all fully isolated and thus arborescences.  The determinant is the sum over the weights of these arborescences.  Since each step of vertex partitioning and isolation introduces an ordering of the vertices (according to the step at which the vertex gets rooted), each arborescence corresponds to a weak ordering of the vertices.  The total number of arborescences is thus the ordered Bell number for the given $n$.  For example, this procedure results in 13 arborescences for the matrix digraph in Fig. \ref{fig:matrix-digraph}.

Though not competitive in time or memory complexity with standard techniques such as LU decomposition for numerical computation of matrix determinants, the sequential and partitioned rooting/vertex-isolation strategies provide interesting ways of factoring matrix determinants and are amenable to recursive and/or parallel programming.  We provide Python codes to illustrate computation of matrix determinants by our graphical vertex-isolation approaches along with some other matrix-tree based methods \cite{Ghosh_Matrix_Digraph_2023}.



\bibliographystyle{plain}
\bibliography{clemson}

\end{document}